\documentclass[11pt]{amsart}
\usepackage{amsmath,amssymb,amscd,amsthm}
\usepackage{amsthm}
\usepackage[margin=1.3in]{geometry}
\usepackage[pagewise]{lineno}

\usepackage{tikz-cd}
\usetikzlibrary{babel}
\usetikzlibrary{decorations.pathmorphing}
\usepackage[margin=1.3in]{geometry}
\usepackage{hyperref}
\hypersetup{hypertex=true,
            colorlinks=true,
            linkcolor=blue,
            anchorcolor=blue,
            citecolor=blue}

\usepackage[all]{xy}

\usepackage{enumerate}
\usepackage{paralist}
\usepackage{amscd}

\newtheorem{thm}{Theorem}[subsection]
\newtheorem{lem}[thm]{Lemma}
\newtheorem{prop}[thm]{Proposition}

\newtheorem{cor}[thm]{Corollary}

\newtheorem*{prob}{\bf Problem}
\theoremstyle{definition}\newtheorem{df}[thm]{Definition}
\theoremstyle{definition}\newtheorem{rem}[thm]{Remark}
\newtheorem{exm}[thm]{\it Example}

\theoremstyle{definition}

\renewcommand{\phi}{\varphi}

\newcommand{\N}{\mathbb{N}}
\newcommand{\Homeo}{\mathcal{H}}

\newcommand{\A}{\mathcal{A}}
\newcommand{\Z}{\mathbb{Z}}
\newcommand{\J}{\mathcal{J}}
\newcommand{\tildeX}{\widetilde{X}}
\newcommand{\X}{\underline{X}}

\newcommand{\KX}{{}_{k_2}X_{l_2}}
\newcommand{\kX}{{}_{k_1}X_{l_1}}
\newcommand{\Kx}{{}_{k_2}[x]_{l_2}}
\newcommand{\kx}{{}_{k_1}[x]_{l_1}}
\newcommand{\jf}{\mathfrak{j}}
\newcommand{\I}{\mathcal{I}}
\newcommand{\D}{\mathcal{D}}
\newcommand{\Sph}{\mathbb{S}}
\newcommand{\La}{\mathcal{L}}

\newcommand{\Raf}{\mathcal{R}_0(\tilde{\alpha})}
\newcommand{\Rbt}{\mathcal{R}_0(\tilde{\beta})}
\newcommand{\C}{\mathbb{C}}
\newcommand{\T}{\mathbb{T}}
\newcommand{\taf}{\tilde{\alpha}}
\newcommand{\tbt}{\tilde{\beta}}

\newcommand{\Aff}{\operatorname{Aff}}
\newcommand{\Pj}{\mathcal{P}}

\newcommand{\id}{\operatorname{id}}

\newcommand{\ep}{\varepsilon}
\newcommand{\F}{{\mathcal F}}

\newcommand{\Lip}{\mathcal{L}}
\newcommand{\G}{\mathcal{G}}

\newcommand{\af}{{\alpha}}
\newcommand{\bt}{{\beta}}

\newcommand{\beq}{\begin{eqnarray}}
\newcommand{\eneq}{\end{eqnarray}}

\begin{document}
\title[The nuc. dim. of Cuntz-Pimsner $C^*$-algebras ass. with minimal shift spaces]{A note on the nuclear dimension of Cuntz-Pimsner $C^*$-algebras associated with minimal shift spaces}

\author{Zhuofeng He, Sihan Wei}

\maketitle
\begin{center}
Research Center for Operator Algebras, 

School of Mathematics and Science, 

East China Normal University, 

Shanghai, China
\end{center}

\begin{abstract} 
For every minimal one-sided shift space $X$ over a finite alphabet, left special elements are those points in $X$ having at least two preimages under the shift operation. In this paper, we show that the Cuntz-Pimsner $C^*$-algebra $\mathcal{O}_X$ has nuclear dimension $1$ when $X$ is minimal and the number of left special elements in $X$ is finite. This is done by describing concretely the cover of $X$ which also recovers an exact sequence, discovered before by T. Carlsen and S. Eilers.
\end{abstract}

\section{Introduction}\label{sec:1}
The Cuntz-Pimsner $C^*$-algebra $\mathcal{O}_X$ is an invariant of conjugacy associated to any shift space $X$. This interplay between shift spaces and $C^*$-algebras starts from the study of the $C^*$-algebra $\mathcal{O}_A$ of a two-sided shift of finite type represented by a $\{0,1\}$-matrix $A$ in a canonical way, see \cite{CK}, in which the associated $C^*$-algebra is originally called a {\it Cuntz-Krieger algebra}. In the next thirty years, the $C^*$-algebra $\mathcal{O}_X$, to every shift space $X$, is constructed and studied in \cite{BCE}, \cite{C}, \cite{C2}, \cite{CE}, \cite{CM}, \cite{KMW}, \cite{KM}, \cite{M}, \cite{M2}, \cite{M3} by several authors (for example, K. Matsumoto, S. Eilers,  T. Carlsen, K. Brix and their collaborators, to name a few), but in different manners for their own uses. We additionally remark that the associated $C^*$-algebra considered in the paper is first defined by T. Carlsen in \cite{C2} using a Cuntz-Pimsner construction, which is why we call it a {\it Cuntz-Pimsner $C^*$-algebra}, as is also pointed out in \cite{BC}.

Among these approaches, the cover $(\tilde{X},\sigma_{\tildeX})$, of a one-sided shift space $X$, is a dynamical system constructed by T. Carlsen in \cite{BC}, and used to define the $\mathcal{O}_X$ as the full groupoid $C^*$-algebra of $\mathcal{G}_{\tilde{X}}$. In particular, the reason why Carlsen considers the groupoid $C^*$-algebra of the cover but not the shift space $X$ itself is that every such cover defines a dynamical system whose underlying map is a local homeomorphism, while this is not always the case for a one-sided shift. Actually, a one-sided shift on an infinite space is a local homeomorphism if and only if it is of finite type, as in \cite{P}.

In \cite{C3}, it is shown that for every shift space $X$ with the property (*), there is a surjective homomorphism $
\rho: \mathcal{O}_X\to C(\underline{X})\rtimes_{\sigma}\Z$,
which sends the diagonal subalgebra $\mathcal{D}_X$ onto the canonical commutative $C^*$-subalgebra $C(\underline{X})$, with $\underline{X}$  the corresponding two-sided shift space of $X$ and $\sigma$ the natural two-sided shift operation. Besides, if $X$ has the property (**), then 
\[{\rm ker}\rho\cong\mathbb{K}^{{\bf n}_X},\]
where ${\bf n}_X$ is a positive integer related to the structure of the left special elements in $X$, namely, the number of right shift tail equivalence classes of $X$ containing a left special element. Consequently, for every minimal shift space $X$, if it has the property (**), which is equivalent to $X$ having finitely many left special elements, then its Cuntz-Pimsner $C^*$-algebra $\mathcal{O}_X$ is an extension of a unital simple A$\T$-algebra by a finite direct sum of the compact operators. Also note that this extension makes $\mathcal{O}_X$ falls into a class of $C^*$-algebras considered by H. Lin and H. Su in \cite{LS}, called the direct limits of generalized Toeplitz algebras.

In \cite{B}, K. Brix considers the $C^*$-algebra $\mathcal{O}_\af$ of a one-sided Sturmian shift $X_\alpha$ for $\af$ an irrational number, by describing the cover of $X_\af$. In particular, he proves that the cover $\tildeX_\af$ of $X_\af$ is a union of the two-sided Sturmian shift $\underline{X_\af}$ and a dense orbit consisting of isolated points. The unique dense orbit corresponds to the unique point $\omega_\af$ in $X_\af$ which has two preimages under the shift operation. This is the first concrete description of covers of non-sofic systems, whereas the cover of a sofic system is a specific class of shifts of finite type. We remark here that the uniqueness of $\omega_\af$ benefits from the well-known fact that $X_\af$ has the smallest complexity growth for shift spaces with no ultimately periodic points: $p_{X}(n)=n+1$ for all $n\ge1$. 

There are two corollaries from the concrete description of the cover of a Sturmian system in \cite{B}: one for a reducing of the exact sequence in \cite{C3}  to its simplest form, that is, $\mathcal{O}_\af$ is an extension of $C(\underline{X_\af})\rtimes_\sigma\Z$ by $\mathbb{K}$; one for the precise value of dynamic asymptotic dimension of the associated groupoid. The latter together with the exact sequence make the $\mathcal{O}_\af$ be of nuclear dimension $1$, where the nuclear dimension is a concept that plays a key role in the classification programs for $C^*$-algebras.

In this note, we generalize this interesting approach and show that for every minimal one-sided shift $X$ with finitely many left special elements, the Cuntz-Pimsner algebra $\mathcal{O}_X$ has nuclear dimension $1$. More specifically, with our concrete description, the cover of each such space will be a finite disjoint union: a copy of the corresponding minimal two-sided shift space $\underline{X}$ (induced from the projection limit of the original one-sided shift), and ${\bf n}_X$ dense orbits, each consisting of isolated points.  This also recovers the whole situation of the exact sequence in \cite{C3}. We also hope that with this description, more $K$-information can be read out from the groupoid for many other minimal shifts, such as non-periodic Toeplitz shifts $X$ with lower complexity growth (which is to sufficiently make $X$ have finitely many left special elements, or equivalently, have the property (**)).

Finally, we also want to emphatically point out that there is a large class of minimal shifts for which our results apply, such as those with bounded complexity growth (see Example \ref{3.3.7} for the definition and Proposition \ref{3.3.9} for details). This class of minimal shifts includes minimal Sturmian shifts considered by K. Brix, which are defined to be the minimal shifts associated with irrational rotations; minimal shifts associated with interval exchange transformations, whose complexity functions are known to satisfy $p_X(n+1)-p_X(n)\leq d$ where $d$ is the number of subintervals; minimal shifts constructed from $(p,q)$-Toeplitz words in \cite{CK2}, where $p,q$ are natural numbers and $p|q$, whose complexity functions are shown to be linear; or also minimal shifts associated with a class of translations on $2$-torus in \cite{BST}, whose complexity functions satisfy $p_X(n)=2n+1$, to name a few.
\subsection{Outline of the paper}
The paper is organized as follows. Section \ref{sec:2} will provide definitions, including basic notions of one-sided shift spaces, the corresponding two-sided shift spaces and $C^*$-algebras. In Section \ref{sec:3}, we recall definitions of past equivalence, right tail equivalence, covers and their properties. A couple of technical preparations will also be presented for the later use. Section \ref{sec:4} is devoted to the main body of the paper, in which we give a concrete description to the cover of a minimal shift with finitely many left special elements. We divide the description into three parts: (i) for isolated points in the cover, see Theorem \ref{4.1.1}; (ii) for the surjective factor $\pi_X$, see Theorem \ref{4.2.4} and Theorem \ref{4.2.5}; (iii) for non-isolated points in the cover, see Theorem \ref{4.3.1}. Finally, we conclude our main result for the nuclear dimension of $\mathcal{O}_X$ in Section \ref{sec:5}.
\subsection{Acknowledgements} The authors were partially supported by Shanghai Key Laboratory of PMMP, Science and Technology Commission of Shanghai Municipality (STCSM), grant $\#$13dz2260400 and a NNSF grant (11531003). The first named author was also supported by Project funded by China Postdoctoral Science Foundation under Grant 2020M681221. We would like to thank members of Research Center for Operator Algebras, for providing online discussions weekly during the difficult period of COVID-19. The second named author would also like to thank his advisor Huaxin Lin from whom he is learning a lot as a doctoral student. The authors thank the anonymous referee for many helpful comments and suggestions as well.

\section{Preliminaries}\label{sec:2}
Throughout the paper, we denote by $\N$ the set of nonnegative integers. For a finite set $S$, we will always use $\#S$ to denote its cardinality.
\subsection{Shift spaces}
Let $\A=\{0,1\}$. Endowed with the product topology, the spaces $\A^\Z$ and $\A^\N$ are homeomorphic to the Cantor space, i.e., the totally disconnected compact metric space with no isolated point. Note that $\A^\Z$ and $\A^\N$ can be given the following metrics:
\begin{align*}
\underline{d}(\underline{x},\underline{y})&=\sup\{1/2^N: \underline{x}_k=\underline{y}_k\ {\rm for\ all}\ 0\leq |k|\leq N-1\},\\
d(x,y)&=\sup\{1/2^N: x_k=y_k \ {\rm for\ all}\ 0\leq k\leq N-1\}.
\end{align*}

We use $\A^*$ and $\A^\infty$ to denote the monoid of finite words and the set of infinite one-sided sequences with letters from $\A$, that is,
\[\A^*=\bigsqcup_{n\ge1}\A^n\cup\{\epsilon\},\ \ \A^\infty=\A^{\N},\]
where $\epsilon$ is the unique empty word in $\A^*$. For a word $\mu\in A^*$, we use $|\mu|$ to denote the length of $\mu$ and write $|\mu|=n$ if $\mu\in\A^n$. For the empty word, we usually define $|\epsilon|=0$. Besides, the length of any element $\mu$ in $\A^\infty$ is defined to be $\infty$. Let $\mu\in\A^*$ and $\nu\in\A^*\sqcup\A^\infty$, we say $\mu$ occurs in $\nu$, if there exists $a\in\A^*$ and $b\in\A^*\sqcup\A^\infty$ such that 
\[\nu=a\mu b.\]
If $\mu$ occurs in $\nu$, we also say $\mu$ is a factor of $\nu$.

A full-shift is a continuous map $\sigma: x\mapsto\sigma(x)$ from $\A^\N$ to $\A^\N$ (or $\A^\Z$ to $\A^\Z$) such that 
\[(\sigma(x))_n=x_{n+1}.\]
A one-sided (two-sided, respectively) shift space is a nonempty compact $\sigma$-invariant subspace $X$ of $\A^\N$ (or $\A^\Z$ respectively) together with the restriction $\sigma|_X$. Note that by $\sigma$-invariant, we mean $\sigma(X)\subset X$. Any two-sided shift is a homeomorphism, and any one-sided shift $\sigma: X\to X$ is injective if and only if $X$ is finite. Throughout the paper we will only consider one-sided shifts on infinite compact spaces. 

If $X$ is a shift space, $x\in X$ and $-\infty<n\leq m<\infty$, we define $x_{(n-1,m]}=x_{[n,m+1)}=x_{[n,m]}=x_nx_{n+1}\cdots x_{m}$. We also use $x_{(-\infty, m]}=x_{(-\infty,m+1)}$ or $x_{[n,\infty)}=x_{(n-1,\infty)}$ to denote the natural infinite positive and negative parts of $x$ respectively.

For any two-sided shift space $X$, we use $X_+$ to stand for the corresponding one-sided shift space, that is, $X_+=\{x_{[0,\infty)}: x\in X\}$. If $X$ is a one-sided shift space, then $\underline{X}$ is used in this paper, to denote the inverse limit of the projective system
\[X\stackrel{\sigma}{\leftarrow}X\stackrel{\sigma}{\leftarrow}\cdots\stackrel{\sigma}{\leftarrow}X\stackrel{\sigma}{\leftarrow}\cdots.\]
Note that $\underline{X}$ is a two-sided shift space under a canonical identification. 

For a shift space $X$, its language $\La(X)$ will play a central role, whose elements are those finite words over $\A$ occurring in some $x\in X$. A language uniquely determines a shift space, or in other words, $x\in X$ if and only if any factor $\mu$ of $x$ is an element of $\La(X)$. This fact implies that for any two-sided shift space $Y$, $\sigma(Y)=Y$, and therefore for any one-sided shift space $X$, $\sigma(X)=X$ if and only if $X=(\underline{X})_+$. Any topologically transitive one-sided shift (for the definition of topologically transitivity, see Proposition \ref{3.3.2}) is automatically surjective since its image is a dense compact subset.
\begin{df}
Let $X$ be a one-sided shift space and $x\in X$. We define the {\it forward and backward orbits} of $x$ to be 
\[Orb^+_{\sigma}(x)=\{\sigma^n(x):n\geq0\}\ \ {\rm and}\ \ Orb^-_{\sigma}(x)=\{y\in X: \exists n>0(\sigma^n(y)=x)\}\]
respectively, and the whole orbit of $x$ to be $Orb_\sigma(x)=Orb^+_{\sigma}(x)\cup Orb^-_{\sigma}(x)$.
\end{df}

\subsection{$C^*$-algebras and groupoids}
\begin{df}[cf. \cite{WZ}, Definition 2.1]
Let $A$ and $B$ be $C^*$-algebras. A $*$-homomorphism $\pi: A\to B$ is said to have {\it nuclear dimension} at most $n$, denoted ${\rm dim}_{\rm nuc}(\pi)\leq n$, if for any finite set $\F\subset A$ and $\varepsilon>0$, there is a finite-dimensional subalgebra $F$ and completely positive maps $\psi: A\to F$ and $\phi: F\to B$ such that $\psi$ is contractive, $\phi$ is $n$-decomposable in the sense that we can write
\[F=F^{(0)}\oplus F^{(1)}\oplus\cdots F^{(n)}\]
satisfying $\phi|_{F^{(i)}}$ is c.p.c order zero for all $i$, and for every $a\in\F$, 
\[\|\pi(a)-\phi\psi(a)\|<\varepsilon.\]
The {\it nuclear dimension} of a $C^*$-algebra $A$, denoted ${\rm dim}_{\rm nuc}(A)$, is defined as the nuclear dimension of the identity homomorphism ${\rm id}_A$.
\end{df}
We now recall the definitions of groupoid and its dynamic asymptotic dimension. 
\begin{df}[cf. \cite{SW}, (3.1)]
Let $X$ be a local homeomorphism on a compact Hausdorff space $X$. We then obtain a dynamical system $(X, T)$.  The corresponding {\it Deaconu-Renault Groupoid} is defined to be the set
\[\mathcal{G}_{X}=\{(x,m-n,y)\in X\times \Z\times X: T^m(x)=T^n(y), m,n\in\N\},\]
with the unit space $\mathcal{G}_{X}^0=\{(x,0,x): x\in X\}$ identified with $X$, range and source maps $r(x,n,y)=x$ and $s(x,n,y)=y$, and operations $(x,n,y)(y,m,z)=(x,n+m,z)$ and $(x,n,y)^{-1}=(y,-n,x)$.
\end{df}
By Lemma 2.3 in \cite{BC}, Lemma 3.1 and Lemma 3.5 in \cite{SW}, the groupoid $\mathcal{G}_{\tildeX}$ considered in the paper will all be locally compact, Hausdorff, amenable and $\acute{\rm e}$tale, where $\tildeX$ is the cover of $X$ in the sense of Definition \ref{3.4.3}. They are also principal since all such $\tildeX$ have no periodic point, as is shown in Section \ref{sec:4}.

The {\it Cuntz-Pimsner $C^*$-algebra} $\mathcal{O}_X$ of a one-sided shift space $X$ is defined to be the (full) groupoid $C^*$-algebra $C^*(\mathcal{G}_{\tildeX})$. The {\it diagonal-subalgebra} $\mathcal{D}_X$ is defined to be $C(\tildeX)\subset \mathcal{O}_X$.

Finally we recall the definition of dynamic asymptotic dimension for $\acute{\rm e}$tale groupoids.
\begin{df}[cf. \cite{GWY}, Definition 5.1]
Let $\mathcal{G}$ be an $\acute{\rm e}$tale groupoid. Then $\mathcal{G}$ has {\it dynamic asymptotic dimension} $d\in\N$ if $d$ is the smallest number with the following property: for every open relatively compact subset $K$ of $\mathcal{G}$ there are open subsets $U_0, U_1,\cdots,U_d$ of $\mathcal{G}^0$ that covers $s(K)\cup r(K)$ such that for each $i$, the set $\{g\in K: s(g),r(g)\in U_i\}$ is contained in a relatively compact subgroupoid of $\mathcal{G}$.
\end{df}
It is known that for a minimal $\Z$-action on a compact space, the associated groupoid has dynamic asymptotic dimension $1$, see Theorem 3.1 in \cite{GWY}.

\section{definitions and preparations}\label{sec:3}
From now on, to avoid invalidity or triviality, we only consider infinite one-sided shift space $X$ with $\sigma(X)=X$. We use $\underline{X}$ to denote the associated two-sided shift space.
\subsection{Left special elements and past equivalence}
\begin{df}[cf. \cite{CE}, 2.2]
Let $X$ be a one-sided shift space and $z\in \underline{X}$. We say that $z$ is {\it left special} if there exists $z'\in \underline{X}$ such that $z_{-1}\neq z'_{-1}$ and $z_{[0,\infty)}=z'_{[0,\infty)}$. If $z\in\underline{X}$ is left special, we also say $x=z_{[0,\infty)}$ is {\it left special} in $X$. We use $Sp_l(\X)$ and $Sp_l(X)$ to denote the collections of left special elements in $\X$ and $X$ respectively.

We say $x\in X$ {\it has a unique past} if $\#(\sigma^{k})^{-1}(\{x\})=1$ for all $k\geq1$. Moreover, we say $x\in X$ {\it has a totally unique past} if $\sigma^n(x)$ has a unique past for all $n\geq1$.
\end{df}
It is clear from the definition that for any one-sided shift space $X$ with $\sigma(X)=X$, a point $x\in X$ is left special precisely when $x$ has at least two preimages under $\sigma$, that is, $\#\sigma^{-1}(\{x\})\ge2$. Therefore for any such one-sided shift on an infinite space, left special elements always exist, or $\sigma$ will be injective which implies that  $X$ is finite. It is also immediate that $x$ has a totally unique past if and only if $x\notin Orb_\sigma(\omega)$ for any $\omega\in Sp_l(X)$.
\begin{prop}\label{3.1.2}
Suppose that $Sp_l(X)$ contains no periodic point of $X$. Then $\#Sp_l(\X)<\infty$ if and only if $\#Sp_l(X)<\infty$.
\end{prop}
\begin{proof}
The map $\pi_+: z\mapsto z_{[0,\infty)}$ induces a surjective map from $Sp_l(\X)$ to $Sp_l(X)$. Therefore, if $Sp_l(\X)$ is finite, so is $Sp_l(X)$. 

Now assume that $Sp(\X)$ is infinite. If $Sp(X)$ is finite, then we can take $x\in Sp_l(X)$ with infinitely many preimages in $Sp_l(\X)$ under $\pi_+$. Denote this infinite preimage by $F$. Since $\A$ is finite, the Pigeonhole principle ensures the existence of an infinite subset $F_1\subset F$ such that for every $x\neq y\in F_1$, $x_{[-1,\infty)}=y_{[-1,\infty)}$. Then we  choose $n_1\leq-1$ such that there exists $x^1,y^1\in F_1$ with 
\[(x^1)_{n_1-1}\neq(y^1)_{n_1-1}\ {\rm but}\ x_{[n_1\infty)}=y_{[n_1,\infty)}\ {\rm for\ all\ }x,y\in F_1.\]
This means that there is some $z^1\in F_1$ such that $(z^1)_{[n_1,\infty)}\in Sp_l(X)$. Similarly, choose an infinite subset $F_2\subset F_1$, an integer $n_2\leq n_1-1$ with the same property as the first step and a point $z^2\in F_2$ such that $(z^2)_{[n_2,\infty)}\in Sp_l(X)$. Repeating this procedure, we have a strictly decreasing sequence of negative integers $\{n_k\}_{k\geq1}$ and an infinite sequence $\{z^k\}_{k\geq1}\subset Sp_l(\X)$ with the following property:
\[(z^k)_{[n_k,\infty)}\in Sp_l(X)\ {\rm and}\ (z^k)_{[n_k,\infty)}=(z^{k+1})_{[n_k,\infty)}\ (k=1,2,\cdots).\]
Note that it follows from the latter condition that $(z^k)_{[n_k,\infty)}$ all lie on a single orbit in $X$. Since $Sp_l(X)$ is finite, it has to contain a periodic point, which is a contradiction.
\end{proof}
{\bf Notation.} Let $S\subset X$ be a set and $l\in\N$. We define $S_{[0,l]}$ to be the set whose elements are the prefixes of $x\in S$ of length $l+1$.
\begin{df}[cf. \cite{CE}, 2.4]
Let $X$ be a one-sided shift space and $l\geq1$. For $x\in X$, set
\[P_l(x)=\{\mu\in\La(X): |\mu|=l, \mu x\in X\}=(\sigma^{l})^{-1}(\{x\})_{[0,l-1]}.\]
For $x,y\in X$, we say $x$ and $y$ are {\it $l$-past equivalent} and write $x\sim_l y$, if $P_l(x)=P_l(y)$. In particular, $x$ and $y$ are said to be {\it past equivalent} if $x\sim_ly$ for some $l\geq1$. 

We call $x$ {\it isolated in past equivalence} if there exists $l\geq1$ such that $x\sim_ly$ implies $x=y$.
\end{df}
 If $x\sim_{l+1}y$, then $x\sim_ky$ for all $1\leq k\leq l$. Consequently, if $x$ is isolated in $l$-past equivalent, then $x$ is isolated in $k$-past equivalent for every $k\geq l$.
\begin{lem}\label{3.1.4}
Suppose that $x\in X$ has a unique past. Then for every $l\geq1$, there exists $N\in\N$ such that, whenever $y\in X$ with $y_{[0,N]}=x_{[0,N]}$, $\#P_l(y)=1$.
\end{lem}
\begin{proof}
Assume that there exists $l_0\geq1$ such that for every $n\in\N$, we can always find some $y^n\in X$ with $y^n_{[0,n]}=x_{[0,n]}$ but $\#P_{l_0}(y^n)\ge2$. We are then given a sequence $\{y^n\}_{n\geq0}$ which is easily seen to converge to $x$ as $n\to\infty$. 

Note that the alphabet $\A$ is finite, we now claim that there exists two distinct words $\mu,\nu$ in $\La(X)$ of length $l_0$ such that two sequences of natural numbers $\{n_k\}_{k\ge0}$ and $\{m_k\}_{k\ge0}$ can be chosen, satisfying
\[\mu y^{n_k}\in X\ \ {\rm and}\ \ \nu y^{m_k}\in X.\]
In fact, from the Pigeonhole principle, there is at least one word $\mu$ with $|\mu|=l_0$ such that $\mu$ can be a  prefix of infinitely many $y^n$, say, $y^{n_k}$ for $k\geq1$. However, if $\mu$ is the unique word with such property, then all others in $\La(X)$ with length $l_0$ can only be prefixes of finitely many of $y^n$, and which means that for some natural number $N$, $y^n$ will only have the unique prefix $\mu$ whenever $n\geq N$. This is then a contradiction.

Finally, note that since $y^n\to x$ as $n\to\infty$, every finite word occurring in $\mu x$ and $\nu x$ is an element of $\La(X)$. This proves $\mu x, \nu x\in X$ and hence $x$ does not have a unique past.
\end{proof}

\subsection{Right tail equivalence and $\mathfrak{j}$-maximal elements}

 \begin{df}[cf. \cite{CE}, 2.2 (a slightly different version)]
 Let $x,x'\in X$. The notation $x\sim_{rte}x'$ is used to mean $x$ and $x'$ are {\it right tail equivalent},  in the sense that there exist $M, M'\in \N$ satisfying
 \[\sigma^M(x)=\sigma^{M'}(x').\]
Set $\J_X=Sp_l(X)/\sim_{rte}$. Let $\mathfrak{j}\in\J_X$ and $\omega\in\mathfrak{j}$. We say $\omega$ is {\it $\jf$-maximal} if for any $\omega'\in\jf$, there is an $m\in\N$ such that $\sigma^m(\omega')=\omega$.
 \end{df}
 \begin{prop}\label{3.2.2}
 Suppose that $Sp_l(X)$ is finite and contains no periodic point of $X$. Then every $\jf\in\J_X$ has a unique $\jf$-maximal element. In particular, an element $\omega\in Sp_l(X)$ is $\jf$-maximal if and only if 
 \[\omega\in\jf\ {\rm and}\ \sigma^m(\omega)\notin Sp_l(X)\ {\rm for\ all}\ m\in\N\setminus\{0\}.\]
 \end{prop}
 \begin{proof}
 Let $\eta\in\jf$ be arbitrary. Since $Sp_l(X)$ is finite and contains no periodic point, we can take $K\in\N$ such that $\sigma^K(\eta)\in Sp_l(X)$ but $\sigma^k(\eta)\notin Sp_l(X)$ for all $k\geq K+1$. Denote $\sigma^K(\eta)$ by $\omega$. We prove the proposition by showing that $\omega$ is $\jf$-maximal. 
 
 Let $\omega'\in\jf\setminus\{\omega\}$. Since $\omega\sim_{rte}\omega'$, there are $M,M'\in\N$ such that
 \[\sigma^M(\omega)=\sigma^{M'}(\omega').\]
 Take the minimal nonnegative integer $M$ so that there is $M'\in\N$ with $\sigma^M(\omega)=\sigma^{M'}(\omega')$. If $M>0$, then $\omega_{M-1}\neq w'_{M'-1}$, which means that $\omega_{[M,\infty)}=\omega'_{[M',\infty)}$ is left special. Note that $\omega_{[M,\infty)}=\sigma^{k+M}(\eta)$. However, this contradicts to the assumption that $\sigma^k(\eta)\notin Sp_l(X)$ for all $k\geq K+1$. Consequently, $M=0$, in other words, $\omega=\sigma^{M'-1}(\omega')$. This proves the existence of $\jf$-maximal elements. 
 
 The uniqueness follows directly from the absence of periodic point in $Sp_l(X)$. Finally, the above argument  verifies the second assertion at the same time.
 \end{proof}
 \begin{df}
 Let $X$ be a one-sided shift space with finitely many left special elements.  From now on,  for any $\jf\in\J_X$, we will always denote the unique $\jf$-maximal element by $\omega_\jf$. For every $\jf\in\J_X$, define
 \[U_\jf=\{\omega\in\jf: Orb_\sigma^-(\omega)\cap \jf=\varnothing\}.\]
 \end{df}
 Note that for all $\omega\in U_\jf$, $\af\omega$ has a unique past whenever $\af\omega\in X$ for some $\af\in\A$.

 \begin{lem}\label{3.2.4}
 Suppose $Sp_l(X)$ is finite and contains no periodic point. For every $\omega\in Sp_l(X)$, there is $N\in\N$ such that $\sigma^n(\omega)$ is isolated in $l$-past equivalence for all $l>n\geq N$.
 \end{lem}
 \begin{proof}
 Let $\omega\in Sp_l(X)$. From Proposition \ref{3.2.2}, let $m\in\N$ be such that $w_\jf=\sigma^m(\omega)$ is $\jf$-maximal for some $\jf\in\J_X$. Since $Sp_l(X)$ is finite, there exists $N'\in\N$ with the following property: 
\[{\rm for\ all}\ y,y'\in X, y,y'\in Sp_l(X)\ {\rm and}\ y_{[0,N']}=y'_{[0,N']}\ {\rm implies}\ y=y'.\]
Let $N=N'+m$. Then $\sigma^N(\omega)$ is isolated in $N'+1$-past equivalence, and therefore for every $l>n\geq N$, $\sigma^n(\omega)$ is isolated in $l$-past equivalence as well.
 \end{proof}

 \subsection{Property (*) and (**)}
\begin{df}[cf. \cite{CE}, Definition 3.1]
A one-sided shift space $X$ {\it has property (*)} if for every $\mu\in\La(X)$, there exists $x\in X$ such that $P_{|\mu|}(x)=\{\mu\}$. We will also say $\X$ {\it has property (*)} if $X$ does so.
\end{df}
\begin{prop}\label{3.3.2}
Let $X$ be a one-sided shift space. Suppose that $X$ is topologically transitive, namely, there is a point $x^0\in X$ such that its forward orbit is dense in $X$. If $Sp_l(X)$ is finite and contains no periodic point in $X$, then $X$ has property (*).
\end{prop}
\begin{proof}
Actually, the proof is basically the same as that of Example 3.6 in \cite{CE} for the minimal case, which goes like follows. Since $X$ is transitive, take $x^0\in X$ with a dense forward orbit, which follows that every word in $\La(X)$ occurs in $x^0$. Therefore it suffices to show that, for every word $\mu$ occurring in $x^0$, there exists $y^0$ such that $P_{|\mu|}(y^0)=\{\mu\}$. Now since $x^0$ is a transitive point, $\mu$ appears in $x^0$ infinitely many times. Consider the intersection
\[Orb^+_{\sigma}(x^0)\cap Sp_l(X).\]
Since $Sp_l(X)$ is finite and contains no periodic point, this intersection has to be finite, which means that there exists $N\geq1$ such that $\sigma^{n}(x^0)\notin Sp_l(X)$ for all $n\geq N$. This follows that for all $n\geq N$, $\sigma^n(x^0)$ has only one preimage. Upon taking $L>N+|\mu|$ with $(x^0)_{[L-|\mu|+1,L]}=\mu$, we conclude that $\sigma^{L+1}(x^0)$ has only one preimage of length $|\mu|$, and which is exactly $\mu$.
\end{proof}

\begin{df}[cf. \cite{CE}, Definition 3.2]
Let $X$ be a one-sided shift space with property (*). If, in addition, $Sp_l(X)$ is finite and contains no periodic point in $X$, then we say $X$ {\it has property (**)} 
\end{df}
The Proposition \ref{3.3.2} together with Proposition \ref{3.1.2} implies the following corollary.
\begin{cor}
A transitive one-sided shift space $X$ has property (**) if and only if $Sp_l(X)$ is finite and contains no periodic point. In particular, if $X$ is minimal, then $X$ has property (**) exactly when $Sp_l(X)$ is finite.
\end{cor}

\begin{exm}
Every non-regular Toeplitz shifts has property (*), as is shown in \cite{CE}. 

We now prove that this is the case for every non-periodic Toeplitz shift. The  same notations as in \cite{W} will be used in the following proposition.
\end{exm}
\begin{prop}
Let $\eta$ be a non-periodic Toeplitz sequence. Then $X_\eta$ has property (*).
\end{prop}
\begin{proof}
Let $\mu\in \La(X_\eta)$. Without loss of generality, assume $\eta_{[0, m-1]}=\mu$ for $m=|\mu|$. We show that 
\[P_{|\mu|}(\eta_{[m,\infty)})=\{\mu\}.\]
Suppose $\mu'\in\La(X_\eta)$ with $\mu'\eta_{[m,\infty)}\in X_\eta$. Then $\mu'\eta_{[m,\infty)}$ can be approximated by a sequence $\sigma^{n_k}(\eta)$. Write $\mu=\mu_1\mu_2\cdots\mu_{|\mu|}$. We then note that $\eta_{m-1}=\mu_{|\mu|}$. 

Consider the $p_{m-1}$-skeleton of $\eta$, say, $\tilde{\eta}\in(\A\cup\{\infty\})^\N$.
Then $\tilde{\eta}$ is a periodic sequence with period orbit $\{\tilde{\eta}, \sigma(\tilde{\eta}), \cdots, \sigma^{p_{m-1}-1}(\tilde{\eta})\}$. From the Pigeonhole principle, there is $0\leq l\leq p_{m-1}-1$ such that there exists infinitely many $n_{k_j}\,(j=1,2,\cdots)$ satisfying
\[n_{k_j}-(m-1)\equiv l\ {\rm mod}\ p_{m-1}\]
for some $l\in\{0,1,\cdots,p_{m-1}-1\}$, which follows $(\sigma^{n_{k_j}}(\eta))_n=\eta_{m-1}$ for all $n\in (l+m-1)+p_{m-1}\N$, and therefore
\[(\mu'\eta)_n=\eta_{m-1}\]
for all $n\in (l+m-1)+p_{m-1}\N$. Due to the fact that the $p$-skeleton of a given Toeplitz sequence is  the "maximal" periodic part with the given period, $\tilde{\eta}$ plays the central role. Hence, the assumption that $\mu'\eta_{[m,\infty)}$ and $\mu\eta_{[m,\infty)}$ has a common right infinite section yields that $l=0$. We then conclude that for all $n\in m-1+p_{m-1}\N$, 
\[(\mu'\eta)_n=\eta_{m-1}=\mu_m\]
and, in particular, $\mu'_{m}=\mu_m$. By repeatedly applying this procedure to $m-1,m-2,\cdots,0$, we therefore have $\mu'=\mu$.
\end{proof}

\begin{exm}\label{3.3.7}
Let $X$ be a one-sided shift. The {\it complexity function} $p_X$ is defined on positive integers, which sends every $n\geq1$ to the number of finite words in $\La(X)$ of length $n$. Namely,
\[p_X(n)=\#\{\mu\in\La(X): |\mu|=n\}.\]
We say that $X$ {\it has a bounded complexity growth} if there exists $K>0$ such that
\[p_X(n+1)-p_X(n)\leq K\]
for all $n\geq1$. Then every minimal one-sided shift space with a bounded complexity growth has property (**), as is shown in Proposition \ref{3.3.9}.

\begin{prop}\label{3.3.9}
If $X$ is a minimal one-sided shift space with a bounded complexity growth, then $X$ has property (**).
\end{prop}
\begin{proof}
It suffices to show that $X$ has only finitely many left special elements. Let $K\in\N$ be a growth bound of $X$. We actually have $\#Sp_l(X)\leq K$. 

If not, then we take $K+1$ distinct points $\{\omega^1,\cdots,\omega^{K+1}\}\subset Sp_l(X)$ and an integer $N\in\N$ such that the following $K+1$ finite words
\[\omega^1_{[0,N]}, \omega^2_{[0,N]},\cdots, \omega^{K+1}_{[0,N]}\]
are distinct. Note that these finite words are all of length $N+1$ and each of which can be extended to the left in at least two different ways. This immediately follows that
\[p_X(N+2)-p_X(N+1)\geq K+1,\]
a contradiction. The proposition then follows.
\end{proof}
\end{exm}

\subsection{Covers of one-sided shift spaces}
\begin{df}
We use $\I$ to denote the set $\{(k,l)\in\N\times\N: 1\leq k\leq l\}$ and $\D$ its diagonal $\{(k,k)\in\I: k\geq1\}$. The partial order $\preceq$ on $\I$ is defined by
\[(k_1,l_1)\preceq(k_2,l_2)\Leftrightarrow (k_1\leq k_2)\wedge (l_1-k_1\leq l_2-k_2).\]
\end{df}
For the later use, we prove a lemma first.
\begin{lem}\label{3.4.2}
Let $\mathcal{F}\subset \I$ be an infinite set. Then $\F$ has an infinite subchain, or in other words, an infinite totally ordered subset of $\mathcal{F}$.
\end{lem}
\begin{proof}
Take $(k_0,l_0)\in\mathcal{F}$ satisfying
\[l_0-k_0=\min\{l-k: (k,l)\in\mathcal{F}\}.\]
Set $\mathcal{F}_0=\{(k,l)\in\mathcal{F}: k\leq k_0\}$. Then $\mathcal{F}_0$ is nonempty. If $\mathcal{F}\setminus\F_0\neq\varnothing$, take then $(k_1,l_1)\in\F\setminus\F_0$ such that
\[l_1-k_1=\min\{l-k: (k,l)\in \F\setminus\F_0\}\]
and set $\F_1=\{(k,l)\in\F\setminus\F_0: k\leq k_1\}$. By repeating this step, we are given a sequence of sets $\{\F_n\}_{n\ge0}$. If each of $\F_n$ is finite, then every $\F_n$ is nonempty, and this is when $\{(k_n,l_n)\}$ becomes an infinite chain. Conversely, if one of $\F_n$ is infinite, say, $\F_N$, then by a partition of $\F_N$ into the following $k_{N+1}-k_N$ parts: 
\[O_k^N=\{(k',l')\in\F_N: k'=k\}\ (k_N<k\leq k_{N+1}),\]
we see that there exists one of $O_k^N$ being infinite, which is a chain as well.
\end{proof}
As in \cite{BC}, for every $(k,l)\in\I$, we define an equivalence relation $\stackrel{k,l}{\sim}$ on $X$ by
\[x\stackrel{k,l}{\sim} x'\ {\rm if}\  x_{[0,k)}=x'_{[0,k)}\ {\rm and}\ P_l(\sigma^k(x))=P_l(\sigma^k(x')).\]
We write $_k[x]_l$ for the $\stackrel{k,l}{\sim}$ equivalence class of $x$ and $_kX_l$ the set of $\stackrel{k,l}{\sim}$ equivalence classes. It is clear that $_kX_l$ is finite. We then have a projective system
\[_{(k_1, l_1)}Q_{(k_2,l_2)}: \KX\ni{}_{k_2}[x]_{l_2}\mapsto {}_{k_1}[x]_{l_1}\in\kX\]
for all $(k_1,l_1)\preceq (k_2,l_2)$.
\begin{df}[cf. \cite{BC}, Definition 2.1]\label{3.4.3}
Let $X$ be a one-sided shift space with $\sigma(X)=X$. By the {\it cover} $\widetilde{X}$ of $X$, we mean the projective limit $\mathop{\lim}\limits_{\longleftarrow}({}_kX_l, {}_{(k,l)}Q_{(k',l')})$. The shift operation $\sigma_{\widetilde{X}}$ on $\widetilde{X}$ is defined so that ${}_k\sigma_{\widetilde{X}}(\tilde{x})_l={}_k[\sigma({}_{k+1}\tilde{x}_l)]_l$ where ${}_{k+1}\tilde{x}_l$ is a representative of a $\stackrel{k+1,l}{\sim}$-equivalence relation class in $\tilde{x}$.
\end{df}
The following sets give a base for the topology of $\tildeX$:
\[U(z,k,l)=\{\tilde{x}\in\tildeX: z\stackrel{k,l}{\sim}{}_k\tilde{x}_l\}\]
for $z\in X$ and $(k,l)\in\I$. It is known that $\sigma_{\tildeX}$ is a surjective local homeomorphism, see \cite{BC} for details.
\begin{df}[cf. \cite{BC}, Definition 2.1]
Let $\pi_X: \tildeX\to X$ to be the map which sends each $\tilde{x}\in\tildeX$ to a point $x=\pi(\tilde{x})$ so that $x_{[0,k)}$ are determined uniquely by $({}_k\tilde{x}_l)_{[0,k)}$ for every $(k,l)\in\I$. Define $\imath_X: X\to\tildeX$ by ${}_k\imath_X(x)_l={}_k[x]_l$ for every $(k,l)\in\I$.
\end{df}
In fact, $\pi_X$ is a continuous surjective factor map from $(\tildeX,\sigma_{\tildeX})$ to $(X,\sigma)$ and $\imath_X$ is an injective map (not necessarily continuous) such that $\pi_X\circ\imath_X={\rm id}_X$.

Before the sequel, we recall the following lemmas.
\begin{lem}[cf. \cite{B}, Lemma 4.2]\label{3.4.5}
Let $X$ be a one-sided shift space. Any isolated point in the cover $\tildeX$ is contained in the image of $\imath_X$ and each fibre $\pi_{X}^{-1}(\{x\})$ contains at most one isolated point. In particular, if $x\in X$ is isolated in past equivalence, then $\imath_X(x)$ is an isolated point in $\tildeX$.
\end{lem}

\begin{lem}[cf. \cite{B}, Lemma 4.4]\label{3.4.6}
Let $X$ be a one-sided shift space. Suppose that $x\in X$ has a unique past, then any $\tilde{x}\in\pi^{-1}_X(\{x\})$ also has a unique past.
\end{lem}
We also note the following lemmas for the later use.
\begin{lem}\label{3.4.7}
Let $X$ be a one-sided shift space with property (**). Suppose that $\omega,\omega'\in X$ are left special elements, $\{(k_m,l_m)\}_{m\ge1}$ is an infinite sequence in $\I$ where $\{k_m\}_{m\ge1}$ is an unbounded sequence with $k_m<k_{m+1}$ for all $m\geq1$, and $\{k_{m+1}-k_m\}_{m\ge1}$ are distinct. Assume that to every $m\geq1$, an integer $0\leq n_{(k_m,l_m)}<l_m$ is associated such that
\[P_{l_m}(\sigma^{n_{(k_m,l_m)}}(\omega'))=P_{l_m}(\omega_{[k_m,k_{m+1})}\sigma^{n_{(k_{m+1},l_{m+1})}}(\omega'))\]
for all $m\geq1$. Then the sequence $\{n_{(k_m,l_m)}\}_{m\ge1}$ is unbounded.
\end{lem}
\begin{proof}
Assume that $\{n_{(k_m,l_m)}\}_{m\ge1}$ is bounded. Then there exists an infinite subsequence $\{(k_{m_i}, l_{m_i})\}_{i\ge1}$ and an $n\in\N$ such that 
\[n_{(k_{m_i},l_{m_i})}=n\ {\rm for\ all\ }i.\]
By passing to the subsequence $\{(k_{m_i}, l_{m_i})\}_{i\ge1}$ and checking the equality of $P_{l_{m_i}}$, we assume, without loss of generality, that $n_{(k_m,l_m)}=n$ for some $n$ and all $m\ge1$. Note that $0\leq n<l_m$. Now that we have
\[P_{l_m}(\sigma^{n}(\omega'))=P_{l_m}(\omega_{[k_m,k_{m+1})}\sigma^{n}(\omega'))\]
 for all $m\ge1$. The condition $k_m<k_{m+1}$ together with the property (**) then infer that 
\[\sigma^n(\omega')\neq \omega_{[k_m,k_{m+1})}\sigma^{n}(\omega')\]
for all $m\ge1$. Also note that since $\{k_{m+1}-k_m\}_{m\ge1}$ are distinct, $\omega_{[k_m,k_{m+1}]}$ are all distinct.

The condition $0\leq n< l_m$ follows that $\#P_{l_m}(\sigma^{n}(\omega'))\ge2$, and therefore 
\[\#P_{l_m}(\omega_{[k_m,k_{m+1})}\sigma^{n}(\omega'))\ge2\]
for all $m\ge1$. This immediately tells us that every $\omega_{[k_m,k_{m+1})}\sigma^{n}(\omega')$ lies on the forward orbit of some left special element. However, since $w_{[k_m,k_{m+1}]}\sigma^n(\omega')$ are distinct points lying in the backward orbit of $\sigma^n(\omega')$, we will then have infinitely many distinct special left elements, which contradicts to the assumption that $X$ has property (**).  
\end{proof}

\begin{lem}\label{3.4.8}
Let $X$ be a minimal one-sided shift with property (**) and $x\in X$. If $x$ has a totally unique past, then $\imath_X(x)\in\tildeX$ is not isolated. Consequently, $\pi_X^{-1}(\{x\})$ contains no isolated point for any $x$ having a totally unique past.
\end{lem}
\begin{proof}
Let $z\in X$ and $(k,l)\in\I$ be so that $\imath_X(x)\in U(z,k,l)$. Then $z\stackrel{k,l}{\sim}x$. Denote $P_l(\sigma^k(x))=\{\mu x_{[0,k)}\}$ with $|\mu|=l-k$. It suffices to find an element $\tilde{x}$ in $\tildeX$ such that $z\stackrel{k,l}{\sim}{}_k\tilde{x}_l$ but ${}_{k'}\tilde{x}_{l'}\stackrel{k',l'}{\nsim}x$ for some $(k',l')\in\I$. 

Let $\omega_{\jf}$ be an arbitrary $\jf$-maximal element for some $\jf\in\J_X$. Since $X$ is minimal, then $\mu x_{[0,k)}$ occurs infinitely many times in the forward orbit of $\omega_\jf$. Take $L\in\N$ sufficiently large so that
\[(\sigma^L(\omega_\jf))_{[0,k)}=\mu x_{[0,k)}.\]
Set $\tilde{x}=\imath_X(\sigma^{L+l-k}(\omega_\jf))$. Then $(\sigma^{L+l-k}(\omega_\jf))_{[0,k)}=x_{[0,k)}$ and $P_l(\sigma^{L+l}(\omega_{\jf}))=\{\mu x_{[0,k)}\}$. This verifies ${}_kx_l\stackrel{k,l}{\sim}z$. However, it is clear that ${}_{k'}\tilde{x}_{l'}\stackrel{k',l'}{\nsim}x$ for some sufficiently large $l'$, since $\sigma^L(\omega_\jf)$ sits in the forward orbit of a left special element.
\end{proof}

\section{The description of covers}\label{sec:4}
In this section, $X$ is always assumed to be a one-sided minimal shift space over the alphabet $\A=\{0,1\}$, having property (**). We will, as before, still use $\X$ to denote the corresponding two-sided shift space. Note that $\X$ is also minimal. We also remark that similar conclusions can be drawn for an arbitrary finite alphabet $\A$, but we instead restrict in this paper to the binary shifts for the simplicity of formulations. 

First we point out the isolated points in $\tildeX$.
\subsection{Isolated points in cover}
\begin{thm}\label{4.1.1}
The set of isolated points in $\tildeX$ is dense in $\tildeX$, which is exactly
\[\imath\left(\bigsqcup_{\jf\in\J_X}Orb_{\sigma}(\omega_\jf)\right),\]
where $\omega_\jf$'s are the unique $\jf$-maximal elements. 
\end{thm}
\begin{proof}
Write $I(\tildeX)$ for the set of isolated points in $\tildeX$. We know from Lemma \ref{3.4.8} that, every isolated point of $\tildeX$ has the form $\imath_X(x)$ for some $x\in X$ which does not have a totally unique past. This means $x\in Orb_{\sigma}(\omega)$ for some $\omega\in Sp_l(X)$. Assume now $\omega\in\jf_0$ where $\jf_0$ is one of right tail equivalence classes. By the definition of $\jf$-maximal elements, we immediately see that $x\in Orb_\sigma(\omega_{\jf_0})$.
This implies the inclusion
\[I(\tildeX)\subset \imath\left(\bigsqcup_{\jf\in\J_X}Orb_{\sigma}(\omega_\jf)\right).\]
Conversely, according to the proof of Lemma \ref{3.2.4}, for every $\jf\in\J_X$, there is a point $z\in Orb_{\sigma}^+(\omega_{\jf})$ isolated in past equivalence, which makes, from Lemma \ref{3.4.5}, $\imath_X(z)$ an isolated point in $\tildeX$. On the other hand, recall that as a local homeomorphism, $\sigma_{\tildeX}$ preserves isolatedness and non-isolatedness, which follows that every point in $\imath(Orb_\sigma(\omega_{\jf}))$ is isolated in $\tildeX$. Since $\jf$ is arbitrary, 
\[\imath\left(\bigsqcup_{\jf\in\J_X}Orb_{\sigma}(\omega_\jf)\right)\subset I(\tildeX).\]
This proves the second assertion. We now show that the set of isolated points in $\tildeX$ is dense. Let $z\in X$ and $(k,l)\in \I$. To show the density, it suffices to take $x\in Orb(\omega_\jf)$ such that $z\stackrel{k,l}{\sim}x$ for some $\jf\in \J_X$. We may assume $z\notin Orb_\sigma(\omega_\jf)$ for all $\jf\in \J_X$. The argument of the existence of such $x$ is then exactly the same as that of Lemma \ref{3.4.8}.
\end{proof}
\begin{cor}
There are precisely $\mathfrak{n}_X$ distinct discrete orbits in $\tildeX$ each of which forms an open invariant subspace $\tildeX$, where $\mathfrak{n}_X=\#\J_X$ is the number of right tail equivalence classes in $Sp_l(X)$. The union of these isolated orbits forms an open dense subset in $\tildeX$. 
\end{cor}
\quad\par

\subsection{The surjective factor $\pi_X$}
Recall that for every $\jf\in\J_X$, the set $U_\jf$ is defined to be 
\[U_{\jf}=\{\omega\in\jf: Orb_{\sigma}^-(\omega)\cap\jf=\varnothing\}.\]
\begin{rem}
Elements in $U_\jf$ are called {\it adjusted} and $\jf$-maximal elements $\omega_{\jf}$ are called {\it cofinal} in \cite{CE}. It is easy to see that $U_\jf$ is nonempty for every $\jf\in\J_X$.
\end{rem}
\begin{lem}\label{4.2.2}
For every $\jf\in\J_X$ and $\omega\in U_\jf$, $\#\pi_X^{-1}(\{\omega\})=3$.
\end{lem}
\begin{proof}
We first show that there are at least three distinct elements in $\pi_X^{-1}(\{\omega\})$. The construction below of these three preimages is similar to that of \cite{B}.

For every $\af\in\{0,1\}$ and $(k,l)\in\I\setminus\D$, let $\mu_{(k,l)}^\af\in\La(X)$ with $|\mu_{(k,l)}^\af|=l-k-1$ be such that $\mu_{(k,l)}^\af\af\omega\in X$. Note that such finite word $\mu_{(k,l)}^\af$ is unique because $\omega\in U_{\jf}$. Now since $X$ has property (**), we can take $x_{(k,l)}^\af\in X$ satisfying
\[P_l(x_{(k,l)}^\af)=\{\mu_{(k,l)}^\af\af\omega_{[0,k)}\}\ \ (\af=0,1).\]
Define ${}_kx^\af_l=\omega_{[0,k)}x_{(k,l)}^\af$. Note that for $\af=0,1$, we have
\[\omega\stackrel{k,l}{\nsim}{}_kx^\af_l\ \ {\rm and}\ \ {}_kx^0_l\stackrel{k,l}{\nsim}{}_kx^1_l.\]
We now define representatives on each $(k,k)\in\D$.  Take $x_{(k,k)}^\af\in X$ with $P_k(x_{(k,k)}^\af)=\{\omega_{[0,k)}\}$.  Let ${}_kx_k^\af=\omega_{[0,k)}x_{(k,k)}^\af$. Now for $\af=0,1$, set $\tilde{x}^\af\in\tildeX$ satisfying
\[{}_k(\tilde{x}^\af)_l=[{}_kx^\af_l].\]
It is clear that $\pi(\tilde{x}^\af)=\omega$ and $\{\imath(\omega), \tilde{x}^0, \tilde{x}^1\}$ are three distinct elements. It is now enough to show that $\tilde{x}^\af$ are well-defined. We will only verify the case for $\af=0$, since the other one is exactly the same. For the simplicity of notations, we drop all the superscripts and abbreviate $\tilde{x}^0$ to $\tilde{x}$, for instance.

{\bf (i)} Let $(k_1,l_1)\preceq(k_2,l_2)$ be indices in $\I\setminus\D$. It is trivial that $(\kx)_{[0,k_1)}=\omega_{[0,k_1)}=(\Kx)_{[0,k_1)}$, so it remains to show that 
\[\{\mu_{(k_1,l_1)}0\omega_{[0,k_1)}\}=P_{l_1}(x_{(k_1,l_1)})=P_{l_1}(\omega_{[k_1,k_2)}x_{(k_2,l_2)}).\]
For every $\nu\in P_{l_1}(\omega_{[k_1,k_2)}x_{(k_2,l_2)})$, $\nu\omega_{[k_1,k_2)}\in P_{k_2-k_1+l_1}(x_{(k_2,l_2)})$, and since $l_1+k_2-k_1\leq l_2$, $\nu\omega_{[k_1,k_2)}$ is the suffix of an element in $P_{l_2}(x_{(k_2,l_2)})$, which follows $\nu=\nu'0\omega_{[0,k_1)}$ where $\nu'$ is the suffix of $\mu_{(k_2,l_2)}$ with length $l_1-k_1-1$. However, as $0\omega$ has a unique past, $\nu'=\mu_{(k_1,l_1)}$. 

{\bf (ii)} Let $(k_1,k_1)\in\D$ and $(k_2,l_2)\in\I\setminus\D$ with $k_1\leq k_2$. We shall confirm that 
\[\{\omega_{[0,k_1)}\}=P_{k_1}(\omega_{[k_1,k_2)}x_{(k_2,l_2)}).\]
Since for every $\nu\in P_{k_1}(\omega_{[k_1,k_2)}x_{(k_2,l_2)}$, $\nu\omega_{[k_1,k_2)}\in P_{k_2}(x_{(k_2,l_2)})$. The inequality $k_2\leq l_2$ infers that  $\nu\omega_{[k_1,k_2)}$ is the suffix of some element in $P_{l_2}(x_{(k_2,l_2)})=\{\mu_{(k_2,l_2)}0\omega_{[0,k_2)}\}$, which has to be $\omega_{[0,k_2)}$. Therefore, $\nu=\omega_{[0,k_1)}$.

{\bf (iii)} The case for $(k_1, k_1), (k_2,k_2)\in\D$ where $k_1\leq k_2$ is quite similar to the case (ii) and hence we omit the verification.

Now that we have shown $\#\pi_X^{-1}(\{\omega\})\geq3$. We next prove that these are exactly the only three elements on this fibre.

Take $\tilde{x}\in\pi^{-1}_X(\omega)$ and write $_k\tilde{x}_l={}_k[\omega_{[0,k)}x_{(k,l)}]_l$ for some $x_{(k,l)}\in X$. 
\quad\par
\quad\par
{\bf Claim.} If there exists $(k_0,l_0)\in\I$ such that $\# P_{l_0}(x_{(k_0,l_0)})=1$, then $\tilde{x}\in\{\tilde{x}^0, \tilde{x}^1\}$.

This is immediate. Suppose $P_{l_0}(x_{(k_0,l_0)})=\{\mu0\omega_{[0,k_0)}\}$, then every $x_{(k',l')}$ with $(k',l')\preceq (k_0,l_0)$ are determined. Also note that for all $(k',l')$ with $(k_0,l_0)\preceq (k',l')$ are also unique determined because $0\omega$ has a unique past. Then $x_{(k,l)}$ are all determined, because $\I$ is directed in the sense that given any two points $(k',l'), (k'',l'')\in\I$, we can always find $(k''',l''')\in\I$ with $(k',l')\preceq(k''',l''')$ and $(k'',l'')\preceq(k''',l''')$.
\quad\par
\quad\par
Now assume that $\# P_{l}(x_{(k,l)})\ge2$ for all $(k,l)\in\I$. We then show that $\tilde{x}=\imath_X(\omega)$, which will finish the proof. Fix any $(k_0,l_0)\in\I$. Note that this leads to the fact that, for every $(k,l)\in\I$ with $(k_0,l_0)\preceq (k,l)$, there exists $\omega_{(k,l)}\in Sp_l(X)$ and integers $0\leq n_{(k,l)}\leq l-1$ such that 
\[x_{(k,l)}=\sigma^{n_{(k,l)}}(\omega_{(k,l)}).\]
The finiteness of $Sp_l(X)$ implies that there is $\omega'\in Sp_l(X)$ and infinitely many $(k_m,l_m)\in\I$ satisfying $(k_0,l_0)\preceq (k_m,l_m)$, $k_m<k_{m+1}$ for all $m\ge1$, $\lim_{m\to\infty}k_m=\infty$,  and 
\[x_{(k_m,l_m)}=\sigma^{n_{(k_m,l_m)}}(\omega'),\ \ \ {\rm for\ all\ }m\ge1.\]
Upon passing to a subsequence, we may assume, according to Lemma \ref{3.4.2}, that $\{(k_m,l_m)\}_{m\ge1}$ is a chain in the sense that $(k_m,l_m)\preceq(k_{m'}, l_{m'})$ whenever $m\geq m'$. Besides, we may assume that $\{k_{m+1}-k_m: m\ge1\}$ are distinct. By the definition of cover, we then have
\[P_{l_m}(\sigma^{n_{(k_m,l_m)}}(\omega'))=P_{l_m}(\omega_{[k_m,k_{m+1})}\sigma^{n_{(k_{m+1},l_{m+1})}}(\omega')).\]
Now Lemma \ref{3.4.7} applies, indicating that $n_{(k_m,l_m)}$is unbounded. Hence we may assume, without loss of generality, that $n_{(k_m,l_m)}\to\infty$ as $m\to\infty$. 

On the other hand, from Lemma \ref{3.2.4}, we can take an $N\in\N$ such that $\sigma^n(\omega')$ is isolated in $l$ past equivalence whenever $l>n\geq N$. Choose $M\in\N$ such that $n_{(k_m,l_m)}>N$ whenever $m>M$.  This follows that 
\[x_{(k_m,l_m)}=\sigma^{n_{(k_m,l_m)}}(\omega')\ {\rm is\ }l-{\rm isolated}\]
whenever $m>M, l>n_{(k_m,l_m)}$. In particular, $x_{(k_m,l_m)}$ is $l_m$-isolated because $l_m>n_{(k_m,l_m)}$. Then, we know from
\[P_{l_m}(\sigma^{n_{(k_m,l_m)}}(\omega'))=P_{l_m}(\omega_{[k_m,k_{m+1})}x_{(k_{m+1}, l_{m+1})})\]
that $x_{(k_m,l_m)}=\sigma^{n_{(k_m,l_m)}}(\omega')=\omega_{[k_m,k_{m+1})}x_{(k_{m+1}, l_{m+1})}$ for all $m>M$. Finally, since $k_1<k_2<\cdots<k_m<k_{m+1}<\cdots$ and $\lim_{m\to\infty}k_m=\infty$, we conclude that for all $m>M$,
\[x_{(k_m,l_m)}=\sigma^{k_m}(\omega),\]
and therefore ${}_{k_m}\tilde{x}_{l_m}={}_{k_m}[\omega]_{l_m}$. Recall that $(k_m,l_m)\succeq(k_0,l_0)$ for every $m$, we then have ${}_{k_0}\tilde{x}_{l_0}={}_{k_0}[\omega]_{l_0}$. Finally, as the above discussion can be applied to every $(k_0,l_0)\in\I$, $\tilde{x}=\imath_X(\omega)$, the lemma follows.
\end{proof}

\begin{df}
Let $x\in X$ and $\{z^m\}_{m\leq0}$ be a sequence in $X$. We say $\{z^m\}_{m\leq0}$ is a {\it directed path terminating at $x$} if $z^0=x$ and $\sigma(z^{m-1})=z^m$ for all $m\leq0$. It is not hard to see that for every one-sided shift space $X$ with $\#Sp_l(X)<\infty$ and every $x\in X$, the number of directed path in $X$ terminating at $x$ is finite. We denote this number by $\mathfrak{d}(x)$.
\end{df}

It immediately follows that for any fixed non-maximal element $\omega\in Sp_l(X)$ and $m_0=\min\{m>0: \sigma^m(\omega)\in Sp_l(X)\}$, we have that  
\[\mathfrak{d}(\omega)=\mathfrak{d}(\sigma(\omega))=\cdots=\mathfrak{d}(\sigma^{m_0-1}(\omega))\ \ {\rm and}\ \ \mathfrak{d}(\omega)=\sum_{\omega'\in \sigma^{-1}(\{\omega\})}\mathfrak{d}(\omega').\]

\begin{thm}\label{4.2.4}
For every $x\in\bigsqcup_{\jf\in\J_X}Orb_{\sigma}(\omega_{\jf})$,
\[\#\pi_X^{-1}(\{x\})=\mathfrak{d}(x)+1.\]
\end{thm}
\begin{proof}
First, we verify the situation for which $x$ has a unique past, that is, $\mathfrak{d}(x)=1$. This could happen for example when $x$ lies in the backward orbit of some $\omega\in U_{\jf}$. Then it is clear that either $0x\in X$ or $1x\in X$. In any case, the procedure of Lemma \ref{4.2.2} defines a non-isolated point in $\pi_X^{-1}(\{x\})$ and an exactly same argument as in Lemma \ref{4.2.2} shows that $\#\pi_X^{-1}(\{x\})=2=\mathfrak{d}(x)+1$.

For the case when $x\in Sp_l(X)$, according to the definition of the integer-valued function $\mathfrak{d}$, there are at most $\mathfrak{d}(x)$ finite prefixes 
\[\mu_{(k,l)}^1, \mu_{(k,l)}^2, \cdots, \mu_{(k,l)}^{\mathfrak{d}(x)}\]
with $|\mu_{(k,l)}^i|=l-k$ for sufficiently large $(k,l)\in \mathcal{I}$ such that $\mu_{(k,l)}^ix\in X$ and moreover, for each pair of $\mu_{(k,l)}^ix$ and $\mu_{(k,l)}^jx$ ($i\neq j$) and every $n\in\mathbb{N}$, $\mu_{(k,l)}^ix\neq\sigma^n(\mu_{(k,l)}^jx)$ and $\mu_{(k,l)}^jx\neq\sigma^n(\mu_{(k,l)}^ix)$. Since $X$ has property $(**)$ as assumed, we can take
\[x_{(k,l)}^1, x_{(k,l)}^2,\cdots, x_{(k,l)}^{\mathfrak{d}(x)}\in X\]
satisfying $P_l(x_{(k,l)}^i)=\{\mu_{(k,l)}^ix_{[0,k)}\}$ for $i=0,1,\cdots, \mathfrak{d}(x)$.

Now, an easy adaption of the procedure in Lemma \ref{4.2.2} defines $\mathfrak{d}(x)$ distinct elements in $\pi_X^{-1}(\{x\})$ such that if $\tilde{x}\in\pi_X^{-1}(\{x\})$ is not one of the points we constructed above, then $\tilde{x}=\imath_X(x)$. This proves that for any left special element $x$ in the whole orbit of any maximal left special element $\omega_{\jf}$,
\[\pi_X^{-1}(\{x\})=\mathfrak{d}(x)+1.\]

Finally, let us consider those elements $x\in \bigsqcup_{\jf\in\J_X}Orb_{\sigma}(\omega_{\jf})$ which are not left special. This divides into the following three cases:

(i) $x$ lies in the backward orbit of some $\omega\in Sp_l(X)$ having a unique past;

(ii) $x$ lies in the forward orbit of some maximal element $\omega_\jf$ for $\jf\in\mathcal{J}_X$;

(iii) there are distinct left special elements $\omega,\omega'$ such that $\omega$ lies in the backward orbit of $x$ and $\omega'$ lies in the forward orbit of $x$.

Note that the case (i) has already been included in the the first paragraph above. For (ii) and (iii), let $\omega$ be  a left special element and
\[m_0(\omega)=\min\{m>0: \sigma^m(\omega)\in Sp_l(X)\}.\]
Without loss of generality, we may say $m_0(\omega)=\infty$ if $\omega$ is a maximal element. We now reach (ii) and (iii) by showing that
\[\#\pi_X^{-1}(\{\omega\})=\#\pi_X^{-1}(\{\sigma(\omega)\})=\cdots=\#\pi_X^{-1}(\{\sigma^s(\omega)\})=\mathfrak{d}(\omega)+1\]
for any $1\leq s<m_0(\omega)$.

For this, write $\pi_X^{-1}(\{\omega\})=\{\imath_X(\omega), \tilde{x}^1, \tilde{x}^2, \cdots, \tilde{x}^{\mathfrak{d}(\omega)}\}$, where $\tilde{x}^1, \tilde{x}^2, \cdots, \tilde{x}^{\mathfrak{d}(\omega)}$ are the elements in $\pi_X^{-1}(\{\omega\})$ constructed above. We then have
\[\pi^{-1}_X(\{\sigma^{s}(\omega)\})=\{\sigma^s_{\tildeX}(\imath(\omega)), \sigma^s_{\tildeX}(\tilde{x}^1), \sigma^s_{\tildeX}(\tilde{x}^2), \cdots, \sigma^s_{\tildeX}(\tilde{x}^{\mathfrak{d}(\omega)})\},\]
It is clear that these $\mathfrak{d}(\omega)+1$ elements are distinct and in the preimage of $\sigma^s(\omega)$. Therefore it suffices to show that there are no more elements in the fibre. Suppose that $\tilde{y}\in \pi^{-1}_X(\{\sigma^{s}(\omega)\})$.  Since $\sigma_{\tildeX}$ is surjective, there exists $\tilde{z}$ such that $\sigma_{\tildeX}^s(\tilde{z})=\tilde{y}$. Take $z=\pi_X(\tilde{z})$.  Since $\pi_X$ is a factor, $\sigma^s(z)=\sigma^s(\omega)$. If $z\neq \omega$, then there exists $1\leq j\leq s<m_0(\omega)$ such that $\sigma^j(\omega)$ is left special, but this contradicts to the minimality of $m_0$. Therefore, $z=\omega$ and hence $\tilde{z}\in\{\imath_X(\omega), \tilde{x}^1, \tilde{x}^2, \cdots, \tilde{x}^{\mathfrak{d}(\omega)}\}$. This shows that there are no more elements in the fibre. Noting that
\[\pi_X^{-1}(\{\sigma^s(\omega)\})=\mathfrak{d}(\omega)+1=\mathfrak{d}(\sigma^s(\omega))+1,\]
(ii) and (iii) follow as desired.
\end{proof}

For the last part of the subsection, we consider those  $z\in X$ having totally unique past.
\begin{thm}\label{4.2.5}
Let $z\in X\setminus\bigsqcup_{\jf\in\J_X}Orb_{\sigma}(\omega_\jf)$. Then $\#\pi_X^{-1}(z)=1$.
\end{thm}
\begin{proof}
Let $\tilde{z}\in\tildeX$ with $\pi_X(\tilde{z})=z$. Let us show that $\tilde{z}=\imath_X(z)$. Write ${}_k\tilde{z}_l={}_k[z_{[0,k)}z_{(k,l)}]_l$. We turn to prove that
\[z\stackrel{k,l}{\sim}z_{[0,k)}z_{(k,l)}\]
for all $(k,l)\in\I$. Obviously they have the same initial sections of length $k$. Therefore, it remains to verify that
\[P_l(z_{[k,\infty)})=P_l(z_{(k,l)}).\]
Write $P_l(z_{[k,\infty)})=\{\mu z_{[0,k)}\}$ where $\mu$ is the unique prefix of length $l-k$. We turn to show the following claims to finish the proof.

{\bf Claim 1.} $\mu z_{[0,k)}z_{(k,l)}\in X$: Since $z$ has a unique past, so does $\tilde{z}$. Take the unique $\tilde{z}'\in\tildeX$ so that $\sigma_{\tildeX}(\tilde{z}')=\tilde{z}$. Note that this implies
\[\sigma^{l-k}\pi_X(\tilde{z}')=\pi_X\sigma_{\tildeX}^{l-k}(\tilde{z}')=\pi_X(\tilde{z})=z,\]
and hence $\pi_X(\tilde{z}')=\mu z$. Denote $\tilde{z}'={}_k[{}_k\tilde{z}'_l]_l$. We then have $({}_l\tilde{z}'_l)_{[0,l)}=\mu z_{[0,k)}$. On the other hand, 
\[{}_k\sigma_{\tildeX}^{l-k}(\tilde{z}')_l={}_k\tilde{z}_l={}_k[\sigma^{l-k}({}_l\tilde{z}'_l)]_l,\]
which tells us $z_{[0,k)}z_{(k,l)}={}_k\tilde{z}_l\stackrel{k,l}{\sim}\sigma^{l-k}({}_l\tilde{z}'_l)$. Therefore,
\[\mu z_{[0,k)}\in P_l(\sigma^l({}_l\tilde{z}'_l))=P_l(z_{(k,l)}).\]

{\bf Claim 2.}  $\# P_l(z_{(k,l)})=1$: Since $z$ has a totally unique past, $\sigma^k(z)$ has a unique past. By Lemma \ref{3.1.4}, we can choose $N_1\in\N$ with the following property: 
\[{\rm whenever}\ y\in X\ {\rm with}\ y_{[0,N_1]}=\sigma^k(z)_{[0,N_1]},\ \#P_l(y)=1.\]
Set $N=N_1+k+1$.  Since $(N, l+N-k)\succeq (k,l)$, we have 
\[{}_N\tilde{z}_{l+N-k}\stackrel{k,l}{\sim}{}_k\tilde{z}_l,\]
which follows that
\[P_l(\sigma^k({}_N\tilde{z}_{l+N-k}))=P_l(\tilde{z}_{(k,l)}).\]
However, since $\sigma^k({}_N\tilde{z}_{l+N-k})=z_kz_{k+1}\cdots z_{N-1}\tilde{z}_{(N,l+N-k)}$, it has a prefix of length $N-k=N_1+1$, equal to $\sigma^k(z)_{[0,N_1]}$. Therefore, by how we choose $N_1$, we conclude that 
\[\#P_l(\tilde{z}_{(k,l)})=\#P_l(\sigma^k({}_N\tilde{z}_{l+N-k}))=1.\]
This completes the proof.
\end{proof}
\quad\par

\subsection{Non-isolated points in the cover}
\begin{thm}\label{4.3.1}
Let
\[\tilde{\Lambda}_X=\tildeX\setminus\bigsqcup_{\jf\in\J_X}Orb_{\sigma_{\tildeX}}(\imath_X(\omega_\jf))\]
be the non-isolated points in the cover. Then $\tilde{\Lambda}_X\cong\X$, i.e., there is a canonical conjugacy from $(\tilde{\Lambda}_X,\sigma_{\tildeX})$ to $(\X,\sigma)$, where $\X$ is the two-sided shift associated with $X$.
\end{thm}
\begin{proof}
Note that since the set of isolated points is open, $\tilde{\Lambda}_X$ is closed and invariant. We first show that every element of $\tilde{\Lambda}_X$ has a unique past. For this, by Lemma \ref{3.4.6}, we only need to verify that, for any fixed $k>0$, $\omega\in Sp_l(X)$ and $\tilde{z}\in\pi_X^{-1}(\sigma^k(\omega))\setminus\{\imath_X(\sigma^k(\omega))\}$, $\tilde{z}$ has a unique past. 

Denote $z=\sigma^k(\omega)$. Then $\pi_X(\tilde{z})=z$.
Define 
\[m_z=\min\{m>0: \exists\,\omega'\in Sp_l(X)\,(\sigma^m(\omega')=z)\}.\]
Note that because $\omega$ is left special and $z=\sigma^k(\omega)$, $m_z$ is well-defined. Then we claim that the sets
\[E_i=\{\tilde{y}\in\tilde{X}: \sigma_{\tilde{X}}^i(\tilde{y})=\tilde{z}\}\]
are singletons for $i=1,2,\cdots,m_z$. In fact, for $i=1$, if there are $\tilde{y}_1, \tilde{y}_2\in E_1$, then 
\[\sigma\circ\pi_X(\tilde{y_1})=\pi_X\circ\sigma_{\tilde{X}}(\tilde{y}_1)=\pi_X\circ\sigma_{\tilde{X}}(\tilde{y}_2)=\sigma\circ\pi_X(\tilde{y}_2),\]
and therefore $\pi_X(\tilde{y}_1)=\pi_X(\tilde{y}_2)\in\sigma^{-1}(\{z\})$. This means $\tilde{y}_1, \tilde{y}_2\in \pi_X^{-1}(\sigma^{-1}(\{z\}))$ with $\sigma_{\tilde{X}}(\tilde{y}_1)=\sigma_{\tilde{X}}(\tilde{y}_2)$. However, according to the final paragraph of Theorem \ref{4.2.4} and the minimality of $m_z$, the restriction of $\sigma_{\tilde{X}}$ from $\pi_X^{-1}(\sigma^{-1}(\{z\}))$ to $\pi_X^{-1}(\{z\})$ is injective and onto, which means that $\tilde{y}_1=\tilde{y}_2$, and therefore $E_1$ is a singleton. Note that by the minimality of $m_z$, we can clearly apply the same argument to the case $i=2,3,\cdots,m_z$.

For $i=m_z+1$, from the construction in Lemma \ref{4.2.2}, there is a unique element corresponding to the prefix $0$ or $1$. Therefore $E_{m_z+1}$ is a singleton as well. Repeating this procedure and noting that there exists $K>0$ such that $x$ has a unique past whenever $k\geq K$ and $x\in\sigma^{-k}(z)$, we conclude  that $\tilde{z}$ has a unique past.

On the other hand, it is quite clear that $\sigma_{\tilde{X}}$ is a surjective map restricted on $\tilde{\Lambda}_X$, and from which we can then conclude that $\sigma_{\tilde{X}}$ is a homeomorphism from $\tilde{\Lambda}_X$ onto $\tilde{\Lambda}_X$.

Now we construct a map from $\X$ to $\tilde{\Lambda}_X$. This is a natural construction which is similar to that of the Sturmian case. Specifically,

(i) If $x\in\X$ such that $\sigma^k(x_{[0,\infty)})$ has a unique past for all $k\ge0$, we set
\[\Phi(x)=\tilde{x}=\imath_X(x_{[0,\infty)}),\]
where $\imath_X(x_{[0,\infty)})$ is the unique element in $\pi_X^{-1}(x_{[0,\infty)})$ by Theorem \ref{4.2.5}. Explicitly, for every $x=(x_n)_{n\in\Z}\in\underline{X}$, since $\underline{X}$ is an inverse limit, we regard $x$ as a sequence of right infinite words in $X$:
\[x=\{x_{[-n,\infty)}\}_{n\ge1}.\]
Then we have $\sigma_{\tilde{X}}\circ \Phi(x)=\sigma_{\tilde{X}}\circ\imath_X(x_{[0,\infty)})=\imath_X(x_{[1,\infty)})=\Phi\circ\sigma(x)$.

(ii) If $x\in\X$ such that there is some $k\ge0$ making $\sigma^k(x_{[0,\infty)})$ don't have a unique past, since $X$ has property (**), we can choose $K\ge0$ such that every element in $Orb^+_\sigma(\sigma^K(x))$ is not left special anymore. Therefore, it is enough to determine $\Phi(\sigma^K(x))$. By abuse of notation, we denote $\sigma^K(x)$ by $x$. Let $k$ be the smallest natural number such that 
\[x_{[-k,\infty)}\in Sp_l(X).\]
Then there is a unique element in $\pi_X^{-1}(x_{[-k,\infty)})$ corresponding to the prefix $x_{-k-1}\in\{0,1\}$. Now by applying this argument to $x_{[-(k+1),\infty)}$, together with the assumption that $X$ only has finitely many of special elements, we get a unique element $\Phi(x)$ in $\tilde{\Lambda}_X$. Similar to the case (i), It is straightforward to verify that $\sigma_{\tilde{X}}\circ\Phi(x)=\Phi\circ\sigma(x)$ holds naturally.

Finally, to see that $\Phi$ is a homeomorphism, we first notice that since the topology on $\underline{X}$ and $\tilde{\Lambda}_X$ are both generated by the cylinder sets and that $\Phi$ doesn't change any finite prefix of any right infinite word in the sequence $x=\{x_{[-n,\infty)}\}$, $\Phi$ is clearly continuous and injective. For the surjectivity of $\Phi$, let $\tilde{z}\in\tilde{\Lambda}_X$. Since $\tilde{z}$ has a unique past, $\sigma_{\tilde{X}}^{-n}(\tilde{z})$ is well-defined for all $n\ge1$. Then define a sequence $z$ in $X$ by
\[z=\{\pi_X\circ\sigma_{\tilde{X}}^{-n}(\tilde{z})\}_{n\ge1}\]
Since $\sigma\circ\pi_X\circ\sigma_{\tilde{X}}^{-n-1}(\tilde{z})=\pi_X\circ\sigma_{\tilde{X}}\circ\sigma_{\tilde{X}}^{-n-1}(\tilde{z})=\pi_X\circ\sigma_{\tilde{X}}^{-n}(\tilde{z})$, we see that $z$ corresponds to an element in the projective system $X\stackrel{\sigma}{\leftarrow}X$ and defines a point in $\underline{X}$. From the construction above, we immediately have $\Phi(z)=\tilde{z}$. This verifies the surjectivity of $\Phi$. Finally, since $\underline{X}$ is compact and $\tilde{\Lambda}_X$ is Hausdorff, $\Phi$ is a homeomorphism.
\end{proof}

We now close Sect. \ref{sec:4} by summarizing in the following theorem the main results in the section.
\begin{thm}\label{4.3.2}
Let $(X,\sigma)$ be a one-sided minimal shift over $\{0,1\}$ on an infinite space $X$ with finitely many left special elements. Let $\tildeX$ be its cover. Then we have the following.

(1) The set $I(\tildeX)$ of isolated points in $\tildeX$ is a disjoint union:
\[I(\tildeX)=\bigsqcup_{\jf\in\J_X}\imath(Orb_{\sigma}(\omega_\jf))\]
which forms a dense open subset of $\tildeX$;

(2) The subsystem $(\tildeX\setminus I(\tildeX), \sigma_{\tildeX}|_{\tildeX\setminus I(\tildeX)})$ on the set of non-isolated points is  invertible and conjugate to the canonical two-sided shift space $\X$ of $X$;

(3) For every $x\in X\setminus\bigsqcup_{\jf\in\J_X}Orb_{\sigma}(\omega_\jf)$, 
\[\#\pi_X^{-1}(x)=1.\]
Moreover, for every $x\in\bigsqcup_{\jf\in\J_X}Orb_{\sigma}(\omega_{\jf}),$
\[\#\pi_X^{-1}(x)=\mathfrak{d}(x)+1,\]
where $\mathfrak{d}(x)$ is the number of directed path in $X$ terminating at $x$.
\end{thm}
\begin{rem}\label{4.3.3}
Last but not least, since we only consider systems with alphabet $\mathcal{A}=\{0,1\}$ in order to simplify our proofs, we would also like to mention how our results depend on the number of symbols. In fact, all but (3) in Theorem \ref{4.3.2} hold for systems over any finite alphabet $\mathcal{A}$. In fact,  Lemma \ref{4.2.2} may fail even for $\mathcal{A}=\{0,1,2\}$. This is because for a left special element, say $\omega$, we don't know exactly what is the preimage of $\omega$, for it could be any of $\{0\omega, 1\omega\}$, $\{1\omega, 2\omega\}$, $\{0\omega,2\omega\}$ or $\{0\omega, 1\omega, 2\omega\}$. On the other hand, we see that the proof of (1) and (2) have nothing to do with the number of symbols.
\end{rem}

\section{A commutative diagram and the nuclear dimension}\label{sec:5}
We conclude our main result in this short section, concerning the nuclear dimension of the Cuntz-Pimsner $C^*$-algebra $\mathcal{O}_X$ associated to every minimal one-sided shift over an infinite space $X$ with finitely many special elements.
\begin{thm}\label{5.0.1}
Let $X$ be a one-sided minimal shift space with finitely many  special elements. Then there is a commutative diagram
\[
\xymatrix{
  & 0 \ar[r]&  c_0^{\mathfrak{n}_X} \ar[r]\ar[d]& \mathcal{D}_X \ar[r]\ar[d]& C(\underline{X}) \ar[r]\ar[d]& 0\\
  & 0 \ar[r]&\ar[r] \mathbb{K}^{\mathfrak{n}_X} \ar[r]& \mathcal{O}_X \ar[r]& C(\X)\rtimes_{\sigma}\Z \ar[r] & 0
    }
\]
where the horizontal arrows are short exact, the vertical arrows are inclusions, and $\mathfrak{n}_X$ is the number of right tail equivalence classes of left special elements in $X$. Besides, the Cuntz-Pimsner algebra $\mathcal{O}_X$ has nuclear dimension $1$.
\end{thm}
\begin{proof}
It suffices to show the exact sequence on the second row, since $c_0^{\mathfrak{n}_X}$ corresponds to the abelian $C^*$-algebra of the space of $\mathfrak{n}_X$ discrete orbits and the commutativity of the diagram is induced by $\pi_X$. From the description of the cover $\tildeX$, the unit space of its groupoid $\mathcal{G}_{\tildeX}$ decomposes into two parts:
\[\mathcal{G}_{\tildeX}^0=\tilde{\Lambda}_X\bigsqcup\left(\bigsqcup_{\jf\in\J_X}\imath_X(Orb_{\sigma}(\omega_\jf))\right).\]
In particular, the groupoid restricted to $\tilde{\Lambda}_X$ is isomorphic to $\X\rtimes_\sigma\Z$ by Theorem \ref{4.3.1}, whose $C^*$-algebra is $*$-isomorphic to the crossed product $C(\X)\rtimes_\sigma\Z$, and the groupoid restricted to the open subset $\bigsqcup_{\jf\in\J_X}\imath_X(Orb_{\sigma}(\omega_\jf))$ is the sum of full equivalence relations restricted on each discrete orbit $\imath_X(Orb_\sigma(\omega_\jf))\,(\jf\in\J_X)$, whose $C^*$-algebra is $*$-isomorphic to the direct sum $\mathbb{K}^{\mathfrak{n}_X}$. Then the exactness of the second row follows from Proposition 4.3.2 in \cite{SW}.

For the nuclear dimension of $\mathcal{O}_X$, we first claim that
\[{\bf Claim.}\ \mathcal{G}_{\tildeX}\ {\rm has\ dynamic\ asymptotic\ dimension\ }1.\]
To see this, let $K$ be an open relative compact subset of $\mathcal{G}_{\tildeX}$. Denote the groupoid restricted on $\tilde{\Lambda}_X=\X$ by $\mathcal{G}_{\tilde{\Lambda}}$. It has already been verified that $\mathcal{G}_{\tilde{\Lambda}}$ is a minimal reversible groupoid, or in other words, a groupoid of an invertible minimal action on an infinite compact space, which follows that it has asymptotic dimension $1$. Then there are open subsets $\tilde{U}_0, \tilde{U}_1$ of its unit space $\mathcal{G}_{\tilde{\Lambda}}^0$ that cover $s(K\cap \mathcal{G}_{\tilde{\Lambda}})\cup r(K\cap \mathcal{G}_{\tilde{\Lambda}})$, and the set
\[\{g\in K\cap \mathcal{G}_{\tilde{\Lambda}}: s(g), r(g)\in \tilde{U}_i\}\]
is contained in a relatively compact subgroupoid of $\mathcal{G}_{\tilde{\Lambda}}$ for $i=0,1$. Let
\[U_i=\tilde{U}_i\sqcup\left(\bigsqcup_{\jf\in\J_X}\imath_X(Orb_{\sigma}(\omega_\jf))\right).\]
It is clear that $U_i$ are open and cover $s(K)\cup r(K)$. On the other hand, since the right most one is an discrete open set and $K$ is relatively compact, the set $\{g\in K\setminus\mathcal{G}_{\tilde{\Lambda}}: s(g), r(g)\in U_i\}$
is a finite set for $i=0,1$. This implies that the groupoid generated by 
\[\{g\in K: s(g), r(g)\in U_i\}\]
is a relatively compact subgroupoid for $i=0,1$. This shows $\mathcal{G}_{\tildeX}$ has dynamic asymptotic dimension $1$.

Now from Theorem 8.6 of \cite{GWY}, 
\[{\rm dim}_{\rm nuc}(\mathcal{O}_X)\leq 1.\]
However, by the exact sequence and Proposition 2.9 of \cite{WZ}, 
\begin{align*}
1&=\max\{{\rm dim}_{\rm nuc}(\mathbb{K}^{\mathfrak{n}_X}), {\rm dim}_{\rm nuc}(C(\X)\rtimes_{\sigma}\Z)\}\\
&\leq{\rm dim}_{\rm nuc}(\mathcal{O}_X)\\
&\leq{\rm dim}_{\rm nuc}(\mathbb{K}^{\mathfrak{n}_X})+{\rm dim}_{\rm nuc}(C(\X)\rtimes_{\sigma}\Z)+1=2.
\end{align*}
We then conclude that ${\rm dim}_{\rm nuc}(\mathcal{O}_X)=1$. This finishes the proof.
\end{proof}
\begin{rem}
An alternative argument for the last part of Theorem \ref{5.0.1} would just be that, as a $C^*$-algebra of a groupoid associated with a minimal system over an infinite compact metric space, $\mathcal{O}_X$ is not AF. This follows immediately that ${\rm dim}_{\rm nuc}(\mathcal{O}_X)\ge1$.
\end{rem}
\quad\par



\quad\par

Sihan Wei, School of Mathematics and Science, East China Normal University, Shanghai, China

{\em Email address}: {\bf sihanwei2093@yeah.net}
\quad\par
\quad\par
Zhuofeng He, Research Center for Operator Algebras, East China Normal University, Shanghai, China

{\em Email address}: {\bf zfhe@math.ecnu.edu.cn}


\end{document}